\documentclass{article}
\usepackage[utf8]{inputenc}
\usepackage[english]{babel} 
\usepackage{amsmath} 
\usepackage{dsfont}
\usepackage{centernot}
\usepackage{amssymb}
\usepackage[a4paper,top=2.54cm,bottom=2.0cm,left=2.0cm,right=2.4cm]{geometry}
\usepackage{mathtools}
\usepackage{amsthm}
\usepackage{mathrsfs}
\usepackage{enumitem}
\usepackage{xcolor}
\usepackage{hyperref}
\usepackage{bbm}
\usepackage{comment}
\usepackage{indentfirst}
\renewcommand{\subset}{\subseteq}

\newenvironment{change}{\textcolor{black}}{}

\DeclareMathOperator{\cl}{cl}
\DeclareMathOperator{\ran}{ran}
\DeclareMathOperator{\sep}{Sep}
\DeclareMathOperator{\clop}{Clop}
\DeclareMathOperator{\inter}{int}
\DeclareMathOperator{\val}{val}

\newcommand{\A}{\mathcal{A}}

\newcommand{\psispc}[1]{\Psi\left(#1 \right)}
\makeatletter
\def\moverlay{\mathpalette\mov@rlay}
\def\mov@rlay#1#2{\leavevmode\vtop{%
   \baselineskip\z@skip \lineskiplimit-\maxdimen
   \ialign{\hfil$\m@th#1##$\hfil\cr#2\crcr}}}
\newcommand{\charfusion}[3][\mathord]{
    #1{\ifx#1\mathop\vphantom{#2}\fi
        \mathpalette\mov@rlay{#2\cr#3}
      }
    \ifx#1\mathop\expandafter\displaylimits\fi}
\makeatother

\newcommand{\cupdot}{\charfusion[\mathbin]{\cup}{\cdot}}
\newcommand{\bigcupdot}{\charfusion[\mathop]{\bigcup}{\cdot}}
\newcommand {\pow}{\mathcal P}

\theoremstyle{plain}\newtheorem{teo}{Theorem}[section]
\theoremstyle{plain}\newtheorem{prop}[teo]{Proposition}
\theoremstyle{plain}\newtheorem{lema}[teo]{Lemma}
\theoremstyle{plain}\newtheorem{cor}[teo]{Corollary}
\theoremstyle{definition}\newtheorem{defin}[teo]{Definition}
\theoremstyle{remark}
\theoremstyle{definition}\newtheorem{example}[teo]{Example}
\theoremstyle{plain}\newtheorem{question}[teo]{Question}

\numberwithin{equation}{section}

\title{Almost-normality of Isbell-Mrówka spaces}
\author{Vinicius de Oliveira Rodrigues\\
  \texttt{vinior@ime.usp.br}
  \and
  Victor dos Santos Ronchim\\
 \texttt{vronchim@ime.usp.br}
}
\date{Institute of Mathematics and Statistics\\
University of São Paulo\\[2ex]
\today}
\begin{document}
\newpage

\maketitle

\begin{abstract}
    We explore almost-normality in Isbell-Mrówka spaces and some related concepts. We use forcing to provide an example of an almost-normal not normal almost disjoint family, explore the concept of semi-normality in Isbell-Mrówka spaces, define the concept of strongly $(\aleph_0, <\mathfrak c)$-separated almost disjoint families and prove the generic existence of completely separable strongly $(\aleph_0, <\mathfrak c)$-separated almost disjoint families assuming $\mathfrak s=\mathfrak c$ and $\mathfrak b=\mathfrak c$. We also provide an example of a Tychonoff almost-normal not normal pseudocompact space which is not countably compact, answering a question from P. Szeptycki and S. Garcia-Balan.\vspace{1em}
    
    \noindent\emph{2020 Mathematics Subject Classification:} Primary 54D15, 54D80, 54G20; Secondary 54A35
    
    \noindent\emph{Keywords:} Isbell-Mrówka spaces, almost disjoint families, almost-normal, semi-normal

\end{abstract}

\section{Introduction}
Isbell-Mrówka spaces are topological spaces associated to almost disjoint families. This class of spaces is used to provide examples and counter examples to numerous questions in General Topology, including questions that are initially not related to them.  The topological properties of such spaces often depend on the combinatorial properties of the associated almost disjoint family. We cite the surveys \cite{Hrusak2014} and \cite{HernandezHernandez2018} as references for this field of study.

If $N$ is a countable infinite set such that $N\cap[N]^\omega=\emptyset$, an almost disjoint family (over $N$) is an infinite collection $\mathcal A$ of infinite subsets of $N$ such that for all distinct $a, b \in \mathcal A$, $a\cap b$ is finite. A MAD family (maximal almost disjoint family) is an almost disjoint family which is not properly contained in any other almost disjoint family. By Zorn's Lemma, every almost disjoint family can be extended to a MAD family and it is well known that there exist almost disjoint families of size $\mathfrak c$ \cite{Hrusak2014}. The least cardinality of a MAD family is called $\mathfrak a$, and it is well known that $\mathfrak a\geq \omega_1$.

Given an almost disjoint family $\mathcal A$ over $N$, the Isbell-Mrówka space associated to $\mathcal A$, also known as \change{$\Psi$}-space of $\mathcal A$, and denoted by $\Psi(\mathcal A)$ is the set $N\cup \mathcal A$ with the topology generated by $\{\{n\}: n \in N\}\cup\{\{\{a\}\cup(a\setminus F): a \in \mathcal A, F \in [N]^{<\omega}\}$. It is immediate that $\mathcal A$ is a Hausdorff, locally compact (therefore Tychonoff) not countably compact zero dimensional separable topological space.

In general, $\Psi(\mathcal A)$ does not need to be normal (e.g., if $|\mathcal A|=\mathfrak c$, $|\mathcal A|$ is a closed discrete subspace of size $\mathfrak c$ of the separable space $\Psi(\mathcal A)$, so it is not normal by \change{Jones's} Lemma) but it may be normal, since $\Psi(\mathcal A)$ is metrizable iff $\mathcal A$ is countable. The existence of a uncountable normal Isbell-Mrówka space is independent of the Axioms of ZFC, and is equivalent to the existence of a normal separable non-metrizable Moore space \cite{tall1977set} \cite{qsetsnormality}, \cite{singal1973mildly}.

In this paper we study weakenings of normality on Isbell-Mrówka spaces. We say a topological space is normal iff every two closed disjoint subsets can be separated by open disjoint subsets. Various weakenings of normality have been proposed and studied, such as quasi-normality \cite{zaitsavrussian}, almost-normality \cite{singal1970almost}, mildly-normality \cite{shchepin1972real} and semi-normality \cite{singal1970almost}. In this paper we will focus on the study of almost-normality and semi-normality on Isbell-Mrówka spaces. Recent results regarding the study of some weakenings of normality and Isbell-Mrówka spaces include \cite{paulsergio} and \cite{alshammari2020quasi}.

Given a topological space $X$, a regular closed set of $X$ is a closed set $F$ such that $F=\cl(\inter (F))$, and an open set $U$ is said to be regular open iff $U=\inter( \cl (U))$. We say a topological space $X$ is \textit{almost-normal} iff whenever $F$ is a closed set and $K$ is a regular closed set disjoint from $F$, there exist disjoint open sets $U, V$ such that $F\subseteq U$, $K\subseteq V$. We say that $X$ is \textit{semi-normal} iff for every closed set $F$ and every open set $U$ containing $F$ there exists a regular open set $V$ such that $F\subseteq V\subseteq U$. The following proposition is from \cite{singal1970almost} and can be easily verified:

\begin{prop}[\cite{SingalArya}]\label{semialmost}
A topological space is normal iff it is almost-normal and semi-normal.
\end{prop}

In this paper, we say that an almost disjoint family $\mathcal A$ is [semi, almost]-normal iff $\Psi(\mathcal A)$ is [semi, almost]-normal.

In \cite{paulsergio}, P. Szeptycki, S. Garcia-Balan provided, among several other examples, an example in ZFC, of an almost disjoint family of true cardinality $\mathfrak c$ which is not almost-normal \change{but satisfies weaker separation properties, such as quasi-normality. They also showed that this example could be made MAD in case there exists a MAD family of true cardinality $\mathfrak c$}. They asked the following question (Question 4.2 of \cite{paulsergio}):
\begin{question}\label{question42}Is there an almost-normal not normal almost disjoint family?
\end{question}

Recall $\Psi(\mathcal A)$ is pseudocompact iff $\mathcal A$ is MAD \cite{Hrusak2014}. Thus, if $\mathcal A$ is MAD, it cannot be normal since, as a consequence of Tietze's theorem, pseudocompact normal spaces are countably compact. The authors of \cite{paulsergio} also asked the following (Question 4.3 and 4.4 of \cite{paulsergio}):
\begin{question}\label{question43}Is there an almost-normal MAD family?
\end{question}

\begin{question}\label{question44}Are almost-normal pseudocompact spaces countably compact?
\end{question}

In this paper, we use iterated forcing and a generalization of the notion of $Q$-set to provide a partial answer to Question \ref{question42} (consistently, yes) and answer negatively Question \ref{question44} in ZFC by providing a subspace of $\beta \omega$ which serves as a counter example. Question \ref{question43} remains open.

In \cite{paulsergio}, they also define the concept of strongly $\aleph_0$-separated almost disjoint family, which is related to almost-normal almost disjoint families, as follows: an almost disjoint family $\mathcal A$ (over $N$) is said to be strongly $\aleph_0$-separated iff for every two countable disjoint subsets $\mathcal B, \mathcal C$ of $\mathcal A$ there exists $X\subseteq N$ such that:

\begin{enumerate}[label=(\arabic*)]
    \item For every $a \in \mathcal A$, $a\subseteq^* X$ or $A\cap X=^*\emptyset$;
    \item for every $a \in \mathcal B$, $a\subseteq^* X$;
    \item for every $a \in \mathcal C$, $a\cap X=^*\emptyset$.
\end{enumerate}

They showed that every almost-normal almost disjoint family is strongly $\aleph_0$-separated and showed that, under $CH$, there exist MAD families which are strongly $\aleph_0$-separated. In the last section of this paper we define a stronger concept we call strongly $(\aleph_0, <\mathfrak c)$-separated almost disjoint family and prove that $\mathfrak b=\mathfrak c$ plus $\mathfrak s=\mathfrak c$ implies the generic existence of $(\aleph_0, <\mathfrak c)$-separated completely separable MAD families.

Regarding notation, we define some of the set theoretical topological and cardinal characteristcs concepts as we need them, for undefined concepts we refer (resp.) to \cite{kunen2011set}, \cite{engelking1977general} and \cite{blass2010}.

It is worth mentioning a stronger version of almost-normality, called $\pi$-normality, was proposed \cite{kalantan2008}: a subset of a topological space $X$ is said to be $\pi$-closed if it is a finite intersection of regular closed sets, and $X$ is said to be $\pi$-normal iff whenever $F\subseteq X$ is $\pi$-closed, $K\subseteq X$ is closed and $F\cap K=\emptyset$, there exists disjoint open sets separating $F$ from $K$. However, in \cite{paulsergio} it was proven that almost-normality and $\pi$-normality are equivalent.

\section{A Tychonoff, almost-normal, pseudocompact space which is not countably compact}

In this section we give, in ZFC, a negative answer for Question \ref{question44} by constructing a suitable subspace of $\beta \omega$.

As noted by Kalantan in \cite{kalantan}, extremely disconnected spaces are almost-normal since every regular closed set is a \change{clopen set}, so it can be separated from any set disjoint from it. This fact will be useful to obtain our counterexample.

The following lemma is well known and can be easily proved by the reader. We refer \cite{engelking1977general}.
\begin{lema}\label{lema:dense extreme disconnected}

If $X$ is extremally disconnected and $D\subseteq X$ is a dense subset, then $D$ is also extremally disconnected.
\end{lema}
The following Lemma is also known. We prove it for the sake of completeness.

\begin{lema}\label{lema:sequences implies pseudocompact}
If $D\subseteq X$ is dense and every sequence in $D$ has an accumulation point in $X$, then $X$ is pseudocompact.
\end{lema}

\begin{proof}
Suppose, by contradiction, that $X$ is not pseudocompact. There exists an unbounded continuous function $h:X\rightarrow [0, \infty)\subseteq \mathbb R$. For each $n \in \omega$, let $d(n) \in D\cap h^{-1}[(n, \infty)]$. Then $d:\omega\rightarrow D$ has no limit point $x$, for if it had, we would have $x \in \cl(\{d(n): n\geq m\})$ for every $m \in \omega$, thus, by continuity, $f(x)\geq m$ for every $m \in \omega$, a contradiction. 
\end{proof}

Now we present our example. For the construction, we identify $\beta \omega$ with the space of ultrafilters over $\omega$, where $\mathcal U_n$ is the principal ultrafilter generated by $\{n\}$ for each $n \in \omega$ (and $n$ is identified with $\mathcal U_n$). We write $N=\{\mathcal U_n: n \in \omega\}$. $\omega^*\subseteq \beta \omega$ is the set of free ultrafilters over $\omega$. Given $A\subseteq \omega$, $\hat A$ is the basic \change{clopen set} $\{p \in \beta \omega: A \in p\}$.

\begin{example}\label{example:answer4.4}
There exists a Tychonoff extremely disconnected (thus, almost normal) pseudocompact space which is not countably compact.
\end{example}

\begin{proof}[Construction]
Let $(P_n: n \in \omega)$ be a partition of $\omega$ into pairwise disjoint infinite sets. For each $n \in \omega$, let $p_n$ be a free ultrafilter such that $P_n \in p_n$. Let $F=\{p_n: n \in \omega\}$. $F$ is infinite and discrete since given $n$, $\{p_n\}=F\cap \hat P_n$.

Given $A \in [\omega]^{\omega}$, let $q_A \in \omega^*$ be defined as follows:

\begin{enumerate}[label=(\arabic*)]
    \item If there exists $n\in \omega$ such that $A \in p_n$, let $q_A=p_n$, for any such $n$ (e.g. the least such $n$), or
    \item if for all $n \in \omega$ $A\notin p_n$, let $q_A \in \omega^*$ be any free ultrafilter such that $A \in q_A$.
\end{enumerate}

In any case, $A \in q_A$. Let $X=N\cup \{q_A: A \in [\omega]^{\omega}\}$ and notice that, for each $n\in \omega$, $q_{P_n}=p_n$ by (1). Hence, $F\subseteq X$.

$X$ is a dense subspace of $\beta \omega$ (since it contains $N$) and by Lemma~\ref{lema:dense extreme disconnected}, $X$ is extremely disconnected. In particular, $X$ is also almost normal.

\underline{$X$ is pseudocompact:} since $N$ is dense in $X$, it suffices to see that every sequence $f:\omega\rightarrow N$ has an accumulation point. By passing to a subsequence, we can suppose $f$ is either constant or injective. Constant sequences converge, so suppose $f$ is injective. Let $g:\omega\rightarrow \omega$ be such that $f(n)=\mathcal U_{g(n)}$. Let $A=\ran (g)$. We claim $q_A$ is an accumulation point of $f$. Given a basic nhood $\hat B \ni q_A$, we know $B\cap A \in q_A$ is infinite, so it follows that $g^{-1}[A\cap B]\subseteq\{n \in \omega: f(n) \in \hat B\}$ is also infinite. Since $B$ is arbitrary, the proof is complete.

\underline{$X$ is not countably compact:} we know $F$ is an infinite discrete subspace of $X$ (since it is in $\beta \omega$). Thus, it suffices to show that $F$ is closed in $X$. We show $X\setminus F$ is open in $X$. Clearly, every point of $N$ is in the interior of $X\setminus F$ since $N$ is open. If $A \in [\omega]^{\omega}$ and $q_A \notin F$, then (2) holds, so $q_A \in\hat A$ and $F\cap \hat A=\emptyset$, that is, $q_A\in X\cap \hat A\subseteq X\setminus F$.
\end{proof}

The space constructed in Example~\ref{example:answer4.4} answers negatively the Question 4.4 from \cite{paulsergio}. It is worth mentioning that Isbell-Mrówka spaces are never extremally disconnected, so a similar strategy cannot be employed when trying to address Questions 4.2 and 4.3.

\section{Equivalences for almost-normality in \texorpdfstring{$\Psi(\A)$}{Psi(A)}}
In this section we start to explore the notion of almost-normality in the realm of Isbell-Mrówka spaces. In particular, we aim to provide some characterizations for ``$\mathcal A$ is almost-normal''. In order to do so, we will use the well known notion of a partitioner of an almost disjoint family. As in the introduction, $N$ denotes an infinite countable set for which $N\cap[N]^\omega=\emptyset$.

\begin{defin}
Let $\A$ be an almost disjoint family (over $N$). We say that $X\subseteq N$ is a partitioner for $\A$ if for each $a\in \A$, $a\subseteq^* X$ or $a\cap X =^* \emptyset$. 

We say that a partitioner $X$ for $\mathcal A$ is a partitioner for $\mathcal B,\mathcal C\subseteq\A$ if $b\subseteq^* X$ and $c\cap X=^*\emptyset$ for each $b\in \mathcal B$ and $c\in \mathcal C$.
\end{defin}
The main motivation for our equivalences is the following classical result. We give \cite{Hrusak2014} and \cite{qsetsnormality} as references.
\begin{prop}
Let $\mathcal A$ be an almost disjoint family. Then $\A$ is normal if, and only if, for all $\mathcal B\subseteq \mathcal A$, $\mathcal B$ and $\mathcal A\setminus \mathcal B$ can be separated by disjoint open sets of $\psi(\mathcal A)$.
\end{prop}

Recall there is a one to one correspondence between the \change{clopen subsets} of $\Psi(\mathcal A)$ and the partitioners of $\mathcal A$ which can be defined as follows: for each $X\subseteq N$ consider $\mathscr{B}_X=\{a \in\A: a\subseteq^* X\}$ and $\mathscr{C}_X=\{a\in\A: a\cap X=^*\emptyset\}$. It follows that:
\begin{itemize}
    \item $\mathscr{B}_X\cup X$ and $\mathscr{C}_X\cup (N\setminus X)$ are disjoint open subsets of $\psispc{\A}$;
    \item If $X$ is a partitioner for $\psispc{\A}$, then $\A= \mathscr{B}_X\cupdot \mathscr{C}_X$ and $\psispc{\A} = (\mathscr{B}_X\cup X)\bigcupdot (\mathscr{C}_X\cup (N\setminus X))$ is union of \change{clopen subsets}. 
\end{itemize}

Then it easily follows that:
\begin{lema}\label{lemma: sepclop}
If $\A$ is an almost disjoint family, then $F:\clop(\psispc{\A})\longrightarrow \{X\subseteq N: X \text{ is a partitioner for } \mathcal A\}$, defined by $F(W)= W\cap N$, is a bijective function, with inverse given by $F^{-1}(X)=\mathscr{B}_X \cup X$.
\end{lema}

The regular closed subsets of $\Psi(\mathcal A)$ are easily characterized by the following proposition:
\begin{lema} Let $\A$ be an almost disjoint family. Then $F\subseteq \Psi(\A)$ is a regular closed set iff there exists $W\subseteq N$ such that $F=\cl(W)=W\cup\{a \in \mathcal A: |a\cap W|=\omega\}$.
\end{lema}

\begin{proof}
First, notice that given a subset $W$ of $N$, $W\subseteq \inter (\cl(W))$, therefore $\cl(W) \subseteq \cl(\inter(\cl (W)))$, concluding that $\cl (\inter (\cl (W)))=\cl (W)$ since $\cl (W)$ is closed. Also, it is easy to see that $\cl(W)=W\cup\{a \in \mathcal A: |a\cap W|=\omega\}$. This proves the ``if'' clause.

To prove the ``only if'', suppose $F$ is a regular closed set. Let $W=F\cap N$. \change{It is straightforward to verify that $F=\cl(W)$}.
\end{proof}

\begin{lema}\label{lemma: separa por clopen}
If $\psispc{\A}$ is almost-normal, then for all $B,C\subseteq \A$, $B\cap C=\emptyset$, the following holds:
\[
B \text{ and } C \text{ are separated by open sets} \iff B \text{ and } C \text{ are separated by \change{clopen sets}.}
\]
\end{lema}
\begin{proof}
\change{If $B$ and $C$ are separated by disjoint open sets $U_B$ and $U_C$, respectively, then $F\doteq \cl_{\psispc{\A}}(U_B\cap N)$ is a regular closed set and $\A\setminus F$ is closed. Since $\A$ is almost-normal, there exists $V,W$ disjoint open subsets of $\psispc{\A}$ such that $F\subseteq V$ and $\A\setminus F\subseteq W$. One can verify that $X\doteq V\cap N$ is a partitioner for $\mathcal A$, then by Lemma~\ref{lemma: sepclop}, $\mathscr{B}_X\cup X$ and its complement are the desired \change{clopen sets.}}
\end{proof}

From this lemma, it easily follows that $\psispc{\A}$ is normal iff every two disjoint subsets $B,C\subseteq \A$ are separated by \change{clopen sets}. Now we are ready to characterize the almost-normality of Isbell-Mrówka space by using partitioners and \change{clopen sets}:

\begin{teo}\label{teo: equivalents to almost-normal}
If $\A$ is an almost disjoint family then the following are equivalent:
\begin{enumerate}[label=(\arabic*)]
    \item $\psispc{\A}$ is almost-normal;
    \item For each $F$ regular closed set, there exists a partitioner $X$ for $F\cap\A$ and $\A\setminus F$;
    \item For each $F$ regular closed set, there exists a \change{clopen set} $C$ such that $F\cap\A\subseteq C$ and $\A\setminus F \subseteq \psispc{\A}\setminus C$; 
    \item For each $F$ regular closed set, there exists a \change{clopen set} $C$ such that $F\subseteq C$ and $\A\setminus F \subseteq \psispc{\A}\setminus C$;
    \item Closed sets are separated from regular closed sets by \change{clopen sets}.
\end{enumerate}
\end{teo}

\begin{proof}

$(1)\implies(3):$ If $F$ is a regular closed, there exist disjoint open sets $U,V$ such that $F\subseteq U$ and $\A\setminus F\subseteq V$. By Lemma~\ref{lemma: separa por clopen}, $F\cap\A$ and $\A\setminus F$ are separated by \change{clopen sets}.

$(2)\iff(3):$ \change{This is clear by Lemma~\ref{lemma: sepclop}.}

$(3)\implies(4):$ If $F$ is regular closed set of $\psispc{\A}$, let $C$ be a \change{clopen set} such that $F\cap\A\subseteq C$ and $\A\setminus F \subseteq \psispc{\A}\setminus C$, it follows that:
\[
F\subseteq C\cup ( \underbrace{(F\cap N)\setminus C}_{\text{\change{clopen set}}} ) = C\cup (F\cap N)
\]

\change{It is straightforward to verify that  $Y\doteq (F\cap N)\setminus C$ is a clopen set.}

$(4)\implies(5):$ Let $F, K\subseteq \Psi(\A)$ be disjoint closed sets, where $F$ is regular closed. By (4), there exists a \change{clopen set} $C$ such that $F\subseteq C$ and $\A\setminus F\subseteq \Psi(\A)\setminus C$.

Let $C'=C\setminus(K\cap N)$. Clearly, $C'$ is a closed set containing $F$. $K$ is disjoint from $C'$ since $K\cap \A\subseteq \A\setminus F$ is disjoint from $C$. \change{It is straightforward to verify that $C'$ is also open}.

$(5)\implies(1):$ Trivial.
\end{proof}
This characterization will be useful in the next section to provide an example of an almost disjoint family which is almost-normal but not normal (consistently).

\section{An almost-normal family which is not normal}
In this section we partially answer Question \ref{question42} by using iterated forcing to create a model for ZFC+CH which has an almost-normal almost disjoint family which is not normal. We will use the equivalence between (1) and (2) of Theorem \ref{teo: equivalents to almost-normal} and a generalization of the notion of $Q$-set.

Given $X\subseteq 2^\omega$, the almost disjoint family over $N=2^{<\omega}$ induced by $X$ is the family $\A_X=\{ A_x: x\in X\}$, were $A_x=\{ x|_n: n\in\omega\}$ for each $x\in X$ \change{ and for $F\subseteq 2^\omega$, we denote $\hat{F} = \{x|_n : n\in\omega, x\in F \}$}. As in \cite{qsetsnormality}, we say that an uncountable $X\subseteq 2^\omega$ is a $Q$-set iff every subset of $X$ is an $F_\sigma$ of $X$. The following folklore result holds (a proof can be found in Proposition 2.2 of \cite{qsetsnormality}):

\begin{prop}[\cite{qsetsnormality}]\label{prop: A normal iff X Qset}
Given an uncountable $X\subseteq 2^\omega$, $\psispc{\A_X}$ is normal iff $X$ is a $Q$-set.
\end{prop}

In what follows next we give a similar characterization for almost-normal almost disjoint families. For this purpose, we need the following:

\begin{defin}
An almost $Q$-set in $2^\omega$ is an uncountable subset $X\subseteq 2^\omega$ such that for every $W \subseteq 2^{<\omega}$, $[W]_X=\{x \in X: \forall m\in \omega\,  \exists n\geq m\,( x|_n \in W)\}$ (which is $\{x \in X: |A_x\cap W|=\omega\}$) is an $F_\sigma$ in $X$.
\end{defin}

We note that the definition of $[W]_X$ is absolute for transitive models of ZFC.

The next proof \change{can be extrated from the proof of Proposition~\ref{prop: A normal iff X Qset} of \cite{qsetsnormality}. We write it here for the sake of completeness.}

\begin{lema}\label{lema43}
\change{Given an uncountable $X\subseteq 2^{\omega}$ and disjoint subsets $\mathcal B,\mathcal C\subseteq \A_X$ such that $\widetilde{B} =\{x\in X : A_x \in \mathcal B \}$ and $\widetilde{C}=\{x\in X : A_x \in \mathcal C \}$ are $F_\sigma$ sets in $X$, then there exists a partitioner $J\subseteq 2^{<\omega}$ which separates them.}
\end{lema}
\begin{proof}
\change{Write $\widetilde B= \bigcup\limits_{n\in\omega} F_n$ and $\widetilde C = \bigcup\limits_{n\in\omega} G_n$}, where $F_n$ and $G_n$ are closed in $X$. We proceed with a standard shoelace argument, defining $J_0 = \widehat{F}_0$, $K_0 = \widehat{G}_0\setminus \widehat{F}_0$, $J_n = \widehat{F}_n\setminus \big(\bigcup_{i<n} \widehat{G}_i\big)$, $K_n = \widehat{G}_n\setminus \big( \bigcup_{i\leq n} \widehat{F}_i\big)$. Let $J\doteq \bigcup_{n\in\omega} J_n$. It follows that $J\cap K_m = \emptyset$ for all $m\in\omega$ and we prove that $J$ is a partitioner for $\mathcal B$ and $\mathcal C$.

\change{If $A_x\in \mathcal B$, then $x\in \widetilde B$} so there exists a $n\in\omega$ such that $x\in F_n$. Since $\bigcup\limits_{i<n} G_i$ is closed, there exists $k \in \omega$ that $\{f \in 2^\omega: x|_k\subseteq f\}\cap \bigcup\limits_{i<n} G_i =\emptyset$. Hence, $A_x\subseteq^* J_n\subseteq J$. Similarly, if \change{$A_x \in \mathcal C$}, $A_x\cap J=^* \emptyset$.
\end{proof}

\begin{cor}Given an uncountable $X\subseteq 2^{\omega}$, $\mathcal A_X$ is almost-normal iff $X$ is an almost $Q$-set.
\end{cor}

\begin{proof}
$(\Longrightarrow)$ For $W\subseteq 2^{<\omega}$ fixed, consider the regular closed set $F=\cl_{\A_X}(W)$. Since $\psispc{\A_X}$ is almost-normal, by Theorem~\ref{teo: equivalents to almost-normal}, there exists a partitioner $J\subseteq 2^{<\omega}$ for $\A_X\cap F$ and $=\A_X\setminus F$. It follows that:
\[
[W]_X = \{x\in X: |A_x\cap W| =\omega \} = \{x\in X: A_x \in F\} = \{x\in X: A_x\subseteq^* J\} = \bigcup_{m\in\omega} \bigcap_{n\geq m} \underbrace{\{x\in X: x|_n\in J \}}_{\text{closed in $X$}}. 
\]

Hence, $X$ is an almost $Q$-set in $2^\omega$.

$(\Longleftarrow)$ By Theorem~\ref{teo: equivalents to almost-normal} it suffices to show that for every regular closed set $F$, there exists a partitioner for \change{$\mathcal B=\A_X\cap F$} and \change{$\mathcal C=\A_X\setminus F$}.

If $F$ is a regular closed set in $\A_X$, there exists $W\subseteq 2^{<\omega}$ such that $F=\cl_{\A_X}(W)$. Notice that $[W]_X$ is a $G_\delta$ since:
\[
[W]_X = \bigcap_{m\in\omega} \bigcup_{n\geq m} \underbrace{\{x\in X: x|_n \in W. \}}_{\text{open in $X$}}
\]

Since $X$ is almost $Q$-set, it follows that both \change{$\widetilde{\mathcal B}=[W]_X$} and \change{$\widetilde{\mathcal C}=X\setminus [W]_X$} are $F_\sigma$ in $X$, \change{so by Lemma \ref{lema43} there exists a partitioner for $\mathcal B$ and $\mathcal C$.}
\end{proof}

Before providing the forcing example, notice that if $M, N$ are countable transitive models for ZFC and $M\subseteq N$, then for every $X, Y\subseteq 2^\omega$ in $M$ with $Y\subseteq X$, $Y$ $(Y \text{ is an } F_\sigma \text{ of } X)^M\rightarrow  (Y \text{ is an } F_\sigma \text{ of } X)^N$ since countable sets of $M$ are countable sets of $N$, and since closed/open subsets of $X$ in $M$ are closed/open subsets of $X$ in $N$.

Now we are ready for the main result of this section.
\begin{prop}
\change{Suppose that (in the ground model) $X\subseteq 2^{\omega}$ is infinite, and let $\kappa=\mathfrak c$. Then there exists a c.c.c. forcing notion $\mathbb P$ of size $\kappa$ such that in every forcing extension by $\mathbb P$, $\kappa=\mathfrak c$ and $X$ is an almost $Q$-set (thus, by the previous corollary, $\Psi(\mathcal A_X)$ is almost-normal in the extension).}
\end{prop}
\begin{proof}
We will proceed by iterated forcing. For the forcing notation, we adopt the countable transitive approach, where $M$ is a fixed ctm for \change{ZFC. Let $\lambda=|X|$}.

First we study the basic step of the iteration which may be found in \cite{OnQsets}. Given $A\subseteq X$ in $M$, let $P(A, X)$ be the sets of all finite $r \in [\omega\times(2^{<\omega}\cup A)]^{<\omega}$ such that for all $n \in \omega$, $x \in A$ and $s \in 2^{<\omega}$, if $(n, x)\in r$ and $(n, s) \in r$, then $s\not\subseteq x$. We order $P(A, X)$ by $r\leq r'$ ($r$ is stronger than $r'$) iff $r'\subseteq r$. $P(A, X)$ is $\sigma$-centered (thus, c.c.c.) since for all $r, r'$, if $r\cap (\omega\times 2^{<\omega})=r'\cap(\omega\times2^{<\omega})$, $r\cup r' \in P(A, X)$ is a common extension. Also, notice that in $M$, $|P(A, X)|\leq\max\{|X|, \omega\}=\change{\mu}$.

    If $G$ is $P(A, X)$ generic over $M$, consider, for each $n$, the set $U_n=\{x \in X: \exists r \in G\, \exists s \in 2^{<\omega}\, (n, s) \in r \text{ and } s\subseteq x\} \in M[G]$.  Clearly, $U_n$ is an open subset of $X$. Then $A=\bigcup_{n \in \omega}X\setminus U_n$ since the sets $D_y=\{r \in P(A, X): \exists n \in \omega\, (n, y) \in r\}$ and $E^n_x=\{r \in P(A, X): \exists s \in 2^{<\omega}\, s\subseteq x \text{ and } (n, s) \in r\}$ for $x \in X\setminus A$, $y \in A$ and $n \in \omega$ are all dense. Thus, $A$ is a $F_\sigma$ of $X$ in $M[G]$.

    Now we recursively construct, working in $M$ a finite support $\change{\kappa}$-stage iterated forcing construction $(\langle(\mathbb P_\xi, \leq_\xi, \mathbbm 1_\xi): \xi\leq \omega_1\rangle,\langle( \mathring{\mathbb Q}_\xi, \mathring{\leq}_\xi, \mathring{\mathbbm 1_\xi}): \xi< \change{\kappa}\rangle)$. As in \cite{kunen2011set}, if $\zeta, \xi\leq \omega_1$, $i_\zeta^{\xi}$ is the usual complete embedding from $\mathbb P_\zeta$ to $\mathbb P_\xi$. Moreover, if $i:\mathbb P\rightarrow \mathbb Q$ is a complete embedding between forcing posets and $\tau$ is a $\mathbb P$-name, $i_*(\tau)$ is the $\mathbb Q$-name recursively defined as $\{(i_*(\sigma), i(p)): (\sigma, p) \in \tau\}$.

    Fix a function $f$ from $\change{\kappa}$ onto $\change{\kappa\times \kappa}$ such that if $f(\xi)\leq (\zeta, \mu)$, then $\zeta\leq \xi$. We will use $f$ as a bookkeeping device. Each $\mathring{\mathbb Q}_\xi$ will have size $\change{\mu}$ and will be forced by $\mathbb P_\xi$ to have the ccc, therefore for each $\xi$, $\mathbb P_\xi$ will have cardinality at most $\change{\kappa}$ and will have the ccc as well.
    
    Suppose we have constructed $(\langle(\mathbb P_\zeta, \leq_\zeta, \mathbbm 1_\zeta): \zeta\leq \xi\rangle,\langle( \mathring{\mathbb Q}_\zeta, \mathring{\leq}_\zeta, \mathring{\mathbbm 1_\zeta}): \zeta< \xi\rangle)$ for some $\xi<\change{\kappa}$. We must determine $(\mathring{\mathbb Q}_\xi, \mathring{\leq}_\xi, \mathring{\mathbbm 1_\xi})$. Suppose that for each stage $\zeta<\xi$ we have also listed all $\mathbb P_\zeta$-nice names for subsets of $\check \omega$ as $(\tau_\zeta^\mu: \mu<\change{\kappa})$. This is possible since \change{$|\mathbb P_\zeta|\leq \kappa=\kappa^\omega$ and has since $\mathbb P_\zeta$ has the countable chain condition}. List all $\mathbb P_\xi$-nice names for subsets of $\check \omega$ as $(\tau_\xi^\mu: \mu<\change{\kappa})$ as well.
    
    Let $f(\xi)=(\zeta, \mu)$. Since $\zeta\leq \xi$, the name ${(i_{\zeta}^\xi)}_*(\tau_\zeta^\mu)$ is a nice $\mathbb P_\xi$-name for a subset of $\check \omega$. Let $( \mathring{\mathbb Q}_\xi, \mathring{\leq}_\xi, \mathring{\mathbbm 1_\xi})$ be such that $\mathbbm 1_{\alpha}\Vdash_{\alpha}( \mathring{\mathbb Q}_\xi, \mathring{\leq}_\xi, \mathring{\mathbbm 1_\xi})\approx P\left(\left[{i_{\zeta}^{\xi}}_*(\tau_\zeta^\mu)\right]_{\check X}, \check X\right)$ and $|\mathring{\mathbb Q}_\xi|\leq \change{\mu}$ (which is possible since $|P(A, X)|=\change{\lambda}$. For instance, we may take $\mathring{\mathbb Q}_\xi$ to be $\change{\check \lambda}$). 
    
    Let $\mathbb P=\mathbb P_{\change{\kappa}}$.
    
    $\mathbb P$ has the ccc and $|\mathbb P|=\change{\kappa}$, \change{$\mathfrak c\leq \kappa$ in any extension by $\mathbb P$ (by counting nice names of subsets of $\check \omega$), so $\mathfrak c=\kappa$ must hold since $\mathbb P$ preserves cardinals since it has the countable chain condition}.
    
    Let $G$ be $\mathbb P$-generic over $M$. We claim $X$ is an almost $Q$-set in $M[G]$. It is uncountable since $\mathbb P$ preserves cardinals. Now let $W$ be a subset of $\omega$ in $M[G]$. \change{Since $\text{cf}(\kappa)^M>\omega$}, There exists $\zeta<\change{\kappa}$ such that $W \in M[G_\zeta]$, where $G_\zeta=(i_{\zeta}^{\change{\kappa}})^{-1}[G]$. There exists $\mu<\change{\kappa}$ such that $W=\val(\tau_\zeta^{\mu}, G_\zeta)$. Let $\xi$ be such that $f(\xi)=(\zeta, \mu)$. Then, since $W=\val({i_{\zeta}^{\xi}}_*(\tau_\zeta^\mu), G_\xi)$. Hence, by the choice of $\mathring{\mathbb Q}_\xi$, $M[G_{\xi+1}]$ contains a $P([W]_X, X)$-generic filter over $M[G_\xi]$, so, in $M[G_{\xi+1}]$, $[W]_X$ is an $F_\sigma$-subset of $X$, hence, the same happens in $M[G]$.
\end{proof}
\begin{cor} The following are relatively consistent with ZFC:
\begin{enumerate}
    \item There exists an almost-normal almost disjoint family which is not normal plus CH.
    \item There exists an almost-normal almost disjoint family of size $\omega_1<\mathfrak c$.
\end{enumerate}
\end{cor}

\begin{proof}
\change{For 1., apply the previous proposition assuming $\mathfrak c=\kappa =\omega_1=|X|$. For 2. assume, for concreteness, that in the ground model, $|X|=\omega_1<\mathfrak c=\omega_2<2^{\omega_1}=\omega_3$.}

\change{In both examples $\Psi(\mathcal A_X)$ is not normal, because if it was normal, then by Jones's lemma we would have (in the extension) that $2^{|\Psi(\mathcal A_X)|}=2^{|X|}=2^{\omega_1}\leq \mathfrak c$, contradicting $2^{\omega_1}=2^{\mathfrak c}>\mathfrak c$ in the first case, and $2^{\omega_1}\geq \omega_3>\mathfrak c$ in the second case. This last $\geq$ inequality holds since $\mathbb P$ preserves cardinals due to the countable chain condition.}
\end{proof}

Thus, it is consistent that there exists an almost-normal almost disjoint family of cardinality $\mathfrak c$ (therefore, not normal), which gives a partial answer to Question \ref{question42}. The  almost disjoint families we constructed are not MAD since no $\mathcal A_X$ is a MAD family (since it can be extended by an infinite antichain of $2^{<\omega}$), so Question \ref{question43} remains fully open.

\section{Semi-normality in Isbell-Mrówka spaces}
In the previous section we have constructed an almost disjoint family which is almost-normal but is not normal by using iterated forcing. We do not know if such an almost disjoint family exists in ZFC. Due to Proposition \ref{semialmost}, a semi-normal almost disjoint family $\mathcal A$ is normal iff $\mathcal A$ is \change{almost-normal}. \change{Thus, the study of semi-normality may come in handy when looking for an almost-normal a.d. family which is not normal in ZFC.}

Semi-normality can be translated in combinatorial terms for Isbell-Mrowka spaces. In the end, it follows that semi-normality is equivalent to a weaker form of separation, which was considered by Dow in \cite{Adow} and Brendle in \cite{Brendle} and we state next:

\begin{defin}
Let $\A$ an almost disjoint family (over $N$) and two subfamilies $\mathcal B, \mathcal C\subseteq \A$, we say that a set $X\subseteq N$ weakly separates $\mathcal B$ and $\mathcal C$ if for all $b\in\mathcal B$ and $c\in\mathcal C$, $|X\cap b|<\omega$ and $|X\cap c|=\omega$. 

We say that $\A$ is weakly separated if for every $\mathcal B\subseteq \A$, the pair $\mathcal B$ and $\A\setminus\mathcal B$ can be weakly separated. 
\end{defin}

Now we are ready to present the combinatorial characterization of semi-normality in Isbell-Mrówka spaces. (2) is a combinatorial property that looks like semi-normality.
\begin{prop}\label{prop: semi-normal equiv frac sep}
 Let $\A$ be an almost disjoint family. The following are equivalent:
 \begin{enumerate}[label=(\arabic*)]
     \item $\psispc{\A}$ is semi-normal;
     \item For each $\mathcal B \subseteq \A$ and each $W\subseteq N$ such that $b\subseteq^* W$ for all $b\in \mathcal B$, there exists $W_0\subseteq W$ satisfying the following:
     \begin{equation}\label{eq: equiv semi-normal}
         \text{for all } a\in\A: \quad a\in \mathcal B \iff a\subseteq^* W_0. 
     \end{equation}
     \item $\A$ is weakly separated.
 \end{enumerate}
\end{prop}
\begin{proof}
$(1)\implies (2)$: Fix $\mathcal B \subseteq \A$ and $W\subseteq N$ such that each $b\in \mathcal B$, $b \subseteq^* W$. Since $\mathcal B$ is closed and $\mathcal B \cup W$ is open, there exists a regular open set such that $\mathcal B\subseteq V \subseteq \mathcal B\cup W$. We claim that \eqref{eq: equiv semi-normal} holds for $W_0\doteq V\cap N$:

If $b\in \mathcal B$, then $b\in V$, so $b\subseteq^* W_0$. On the other hand, if $b\in\A$ is such that $b\subseteq^* W_0$, then $b\in \cl (W_0)= \cl (V)$. Since $b\subseteq^* V$, it follows that $b\in \inter (\cl (V)) = V \subseteq \mathcal B \cup  W$, hence $b\in\mathcal B$.

$(2)\implies (3)$: Fix $\mathcal B\subseteq \A$. By hypothesis, there exists $W_0\subseteq N$ satisfying \eqref{eq: equiv semi-normal}. We claim that $X=N\setminus W_0$ weakly separates $\mathcal B$ and $\A\setminus \mathcal B$. Indeed, if $b\in\mathcal B$ then $b\subseteq^* W_0$, so $b\cap (N\setminus W_0)=b\cap X$ is finite. On the other hand, if $a\in\A\setminus \mathcal B$, then $a\setminus W_0 = a\cap X$ is infinite since $a\nsubseteq^* W_0$.

$(3)\implies(1)$: Let $F$ closed and $U$ open such that $F\subseteq U$, then $F= \mathcal B \cupdot K$ where $\mathcal B = F\cap \A$ and $K= F\cap N$. By hypothesis, there exists $X\subseteq N$ such that for each $b\in\mathcal B$, $|b\cap X|<\omega$ and for each $a\in\A\setminus\mathcal B$, $|a\cap X|=\omega$. Consider $W=U\cap N$ and let $V=\mathcal B \cup K\cup (W\setminus X)$, then $F\subseteq V\subseteq U$. We claim that $V$ is a regular open set.

Clearly, $V$ is open. For the regularity, let $x\in \inter(\cl(V))$. If $x\in N$, then $x\in K\cup (W\setminus X)\subseteq V$. If $x\in\A$, then $x\subseteq^* \cl(V)$ and it follows that $x\subseteq^* K\cup(W\setminus X)$. In the case of $x\cap K$ is infinite,  $x\in F\subseteq V$ since $F$ is closed, otherwise $x\subseteq^* (W\setminus X)$ and it follows that $x\cap X$ is finite, thus $x\in \mathcal B \subseteq V$. 
\end{proof}

Recall a subset $\mathcal A$ of $[\omega]^{\omega}$ is centered iff every finite subset of $\mathcal A$ has infinite intersection, and a pseudointersection of $\mathcal A$ is an infinite set $X$ such that $X\subseteq^*A$ whenever $A \in \mathcal A$. The pseudointersection number $\mathfrak p$ is defined as the least size of a centered subset of $[\omega]^\omega$ which does not admit a pseudointersection. It is well known that $\omega_1\leq \mathfrak p\leq \mathfrak a$ \cite{blass2010}.

In \cite{Brendle}, Brendle observes that $\mathfrak{p}\leq\mathfrak{ap}$, where $\mathfrak{ap}$ is defined as the smallest cardinal $\kappa$ for which there exists an almost disjoint family $\A$ of size $|\A|=\kappa$ that is not weakly separated. The reader can verify this inequality directly by applying the following famous classical result\footnote{The proof can be found in \cite[Theorem 2.15]{kunnen1980} were it is defined a $\sigma$-centered order, so the hypothesis about $MA(\kappa)$ can be replaced by limiting the size of $\mathcal C$ and $\mathcal D$ by $\mathfrak p$.}:

\begin{prop}\label{prop: dibraldinho}
Given $\mathcal C, \mathcal D\subseteq \pow(N)$ such that $\max\{|\mathcal C|, |\mathcal D|\}<\mathfrak p$ and for all $x\in\mathcal D$ and $F\in [\mathcal C]^{<\omega}$, $|x\setminus \bigcup F |=\omega$. Then, there exists $d\subseteq N$, such that for each $x\in\mathcal C$, $|d\cap x|<\omega$ and for each $y\in \mathcal D$, $|d\cap y|=\omega$.
\end{prop}

\begin{cor}
If $\A$ is an almost disjoint family with $|\A|<\mathfrak p$, then $\A$ is semi-normal.
\end{cor}
\begin{proof}
This follows from $\mathfrak p\leq \mathfrak{ap}$ and from Proposition~\ref{prop: semi-normal equiv frac sep}.
\end{proof}

\begin{cor}\label{cor: <p implica normal iff almost-normal}
If $\A$ is an almost disjoint family with $|\A|<\mathfrak p$, then $\A$ is normal iff $\A$ is almost-normal.
\end{cor}
This gives us (consistently) a family of uncountable almost disjoint families for which normality and almost-normality are the same. In particular, if $\mathfrak p>\omega_1$, then no Luzin family is almost-normal. One may ask if, consistently, every almost disjoint family is semi-normal, since if this was the case, every almost-normal almost disjoint family would be normal. However, this is false.

\begin{prop}
If an almost disjoint family $\A$ is semi-normal, then $2^{|\mathcal A|}=\mathfrak c$. In particular, almost disjoint families of cardinality $\mathfrak c$ are not semi-normal.
\end{prop}
\begin{proof}
If $\A$ is semi-normal, then $\mathcal A$ is weakly separated by Proposition \ref{prop: semi-normal equiv frac sep}. But then we may inject $\mathcal P(\mathcal A)$ into $\mathcal P(N)$ by letting, for each $\mathcal B\subseteq \mathcal A$, $X_\mathcal B$ be a subset of $\mathbb N$ such that, for all $ a \in \mathcal A$, $|a\cap X_{\mathcal B}|=\omega$ iff $a \in \mathcal B$.
\end{proof}

By Proposition \ref{prop: semi-normal equiv frac sep}, $\mathfrak{ap}$ is the least cardinality of a non semi-normal almost disjoint family. In \cite{Brendle}, Brendle showed that $\mathfrak{ap}\leq \min\{\text{add}(\mathcal M), \mathfrak{q}\}$, where $\mathfrak q$ is defined as the least cardinality of a subset of $2^\omega$ which is not a $Q$-set and $\text{add}(\mathcal M)$ is the least cardinality of a collection of meager subsets of $\mathbb R$ whose union is not meager. Thus, there exists a non semi-normal almost disjoint family of size $\leq \min\{\text{add}(\mathcal M), \mathfrak q\}$. In particular, this discussion wields the following:

\begin{cor}If $\text{add}(\mathcal M)=\omega_1$, there exists a non semi-normal almost disjoint family of size $\omega_1$.
\end{cor}

\section{Generic existence of \texorpdfstring{$(\aleph_0,<\!\mathfrak c)$}{aleph0-c}- separated MAD families}
In \cite{paulsergio}, it is defined the concept of strongly $\aleph_0$-separated almost disjoint family, which is related to almost-normality. They show that every almost-normal almost disjoint family is strongly $\aleph_0$-separated and that an strongly $\aleph_0$-separated exists under CH. Their paper does not say anything about the converse, which we are going to argue to be consistently false. In this section we modify their technique to weaken the CH hypothesis. First, we define a suitable separation concept.
\begin{defin}
We say that an almost disjoint family $\A$ is strongly $(\aleph_0,<\!\mathfrak c)$-separated iff for every  two disjoint $\mathcal B, \mathcal C\subseteq \A$, with $\mathcal B$ countable and $|\mathcal C|<\mathfrak c$, there exists a partitioner $X\subseteq \omega$ for $\mathcal A$ and $\mathcal B$. 
\end{defin} 

Clearly, every strongly $(\aleph_0,<\!\mathfrak c)$-separated almost disjoint family is strongly $\aleph_0$-separated and these concepts are equivalent under CH.

Now we recall the definitions of $\mathfrak b$ and $\mathfrak s$. If $f, g \in \omega^{<\omega}$, we say that $f<^*g$ iff the set $\{n \in \omega: f(n)\geq g(n)\}$ is finite. An unbounded family in $\omega^\omega$ is a set $\mathcal B\subseteq \omega^\omega$ such that for every $f \in \omega^\omega$ there exists $g \in \mathcal B$ such that $g\not <^* f$. The bounding number $\mathfrak b$ is the smallest cardinality of an unbounded family.

We say $\mathcal S\subseteq \mathcal P(\omega)$ is a splitting family iff for every $X\in [\omega]^\omega$ there exists $A \in \mathcal S$ such that both $X\setminus A$ and $X\cap A$ are infinite. The splitting $\mathfrak s$ is the least size of a splitting family.

It is well known that $\mathfrak p\leq \mathfrak s\leq \mathfrak c$ and that $\mathfrak p\leq \mathfrak b\leq \mathfrak a$, and all inequalities are consistent to be strict \cite{blass2010}.

\begin{lema}\label{lema:separa enumeravel do complemento <b}
Let $\A$ an almost disjoint family with $|\A|<\mathfrak b$. If $\mathcal B\subset \A$ is a countable set, then there exists a partitioner for $\mathcal B$ and $\A\setminus \mathcal B$. In particular, $\A$ is strongly $(\aleph_0,<\!\mathfrak c)$-separated.
\end{lema}
\begin{proof}
Let $\mathcal B=\{b_n:n\in\omega\}$ list all elements of $\mathcal B$. For each $a\in \A\setminus \mathcal B$, consider the function $f_a\in\omega^\omega$ defined by $f_a(n)= \sup(a\cap b_n)$. Since $\mathcal F = \{f_a\in \omega^\omega: a\in\A\setminus \mathcal B\}$ is family of functions with $|\mathcal F|<\mathfrak b$, there exists $g\in\omega^\omega$ such that $f_a<^* g$, for all $a\in\A\setminus \mathcal B$.

Let $X=\bigcup\limits_{n\in\omega} (b_n\setminus g(n))$. We claim that $X$ is a separator for $B$ and $\A\setminus \mathcal B$.

Clearly, we have that $b_n\subseteq^* X$, for all $n\in\omega$. Given $a\in\A\setminus \mathcal B$, since $g>^* f_a$, there exists $k\in\omega$ such that $b_n\setminus g(n)=\emptyset$ for all $n\geq k$, thus $X\cap a=^*\emptyset$.
\end{proof}

\begin{cor}
If $\mathfrak p>\omega_1$, there exists a strongly $(\aleph_0, <\!\mathfrak c)$-separated almost disjoint family which is not almost-normal.
\end{cor}

\begin{proof}
By Corollary~\ref{cor: <p implica normal iff almost-normal} and Lemma~\ref{lema:separa enumeravel do complemento <b}, every non normal almost disjoint family of size $\omega_1$ is strongly $(\aleph_0, <\!\mathfrak c)$-separated and is not almost-normal. \change{So, assuming $p>\omega_1$, any non-normal almost disjoint family of size $\omega_1$ suffices. E.g., a Luzin family.}
\end{proof}

Given an almost disjoint family $\mathcal A$, In what follows, $\mathcal{J}^+(\A)=\{X \in \omega: |\{a \in \mathcal A: |a\cap X|=\omega\}|\geq \omega\}$. An almost disjoint family is said to be completely separable iff for every $A \in \mathcal J^+(\A)$ there exists $a \in \A$ such that $a \subseteq A$. Completely separable almost disjoint families exist in ZFC \cite{galvin}, however, we don't know if completely separable MAD families exist in ZFC even thought we know they exist in most models \cite{Hrusak2014}.
A concept related to completely separability is the true cardinality $\mathfrak c$.

In \cite[Definition~1.2]{genericMAD}, the authors introduce the definition of generic existence of a MAD family in terms of a given property $P$. More precisely, we say that MAD families with a property $P$ exist generically iff all almost disjoint families of size less than $\mathfrak c$ can be extended to a MAD family with the property $P$. In this sense, we have the following result:

\begin{teo}[$\mathfrak b = \mathfrak s =\mathfrak c$] Completely separable MAD families which are strongly $(\aleph_0, <\!\!\mathfrak{c})$-separated exist generically.
\end{teo}
\begin{proof}
Let $\mathcal A'$ be an infinite almost disjoint family of size $\kappa<\mathfrak c$ and write $\mathcal A'=\{a_\gamma: \gamma<\kappa\}$ so that $a_\gamma\neq a_\nu$ whenever $\gamma\neq \nu$.

Let $\{B_\beta\in [\mathfrak c]^\omega :\kappa\leq\beta<\mathfrak c\}$ list all countable subsets of $\mathfrak c$ such that for each $\beta$, $B_\beta\subseteq \beta$ and, for all $B\in [\mathfrak c]^\omega$, $|\{\beta: B_\beta = B\}|= \mathfrak c$ and list $[\omega]^\omega = \{Y_\alpha: \kappa\leq \alpha<\mathfrak c\}$.

We will define recursively almost disjoint families $\A_\alpha$, $a_\alpha\in [\omega]^\omega$ and $X_\alpha \subseteq \omega$, for $\kappa\leq\alpha<\mathfrak c$ such that:
\begin{enumerate}[label=(\arabic*)]
    \item $\A_\beta =\{a_\gamma: \gamma<\beta \}$;
    \item $\mathcal A_\kappa=\mathcal A'$.
    \item $\forall \beta<\mathfrak c: \quad \forall \gamma \in B_\beta$, $a_\gamma\subseteq^* X_\beta$;
    \item $\forall \beta<\mathfrak c: \quad \forall \gamma \in \beta\setminus B_\beta$, $a_\gamma\cap X_\beta =^*\emptyset$;
    \item $\forall \gamma<\mathfrak c: \quad \forall \beta \leq\gamma$, $a_\gamma\subseteq^* X_\beta$ or $a_\gamma\cap X_\beta =^*\emptyset$;
    \item if $Y_\beta\in \mathcal {J}^+(\A_\beta)$, $a_\beta\subseteq Y_\beta$ is an infinite subset.
    \item $\forall \eta<\gamma<\mathfrak c: \quad a_\eta\cap a_\gamma$ is finite.
\end{enumerate}

Fix $\alpha<\mathfrak c$ and suppose that $X_\beta$ and $a_\beta$ are defined for $\kappa\leq \beta<\alpha$. Since $B_\alpha$ is countable and $|\alpha|<\mathfrak c=\mathfrak b$, using Lemma~\ref{lema:separa enumeravel do complemento <b}, let $X_\alpha\subseteq\omega$ be a partitioner for $\{a_\xi: \xi \in B_\alpha\}$ and $\{a_\xi: \xi \in \alpha \setminus B_\alpha\}$.

To define $a_\alpha$, notice that since $|\alpha|<\mathfrak a$ there exists an infinite $Y\subseteq \omega$ almost disjoint from $a_\beta$, for all $\beta<\alpha$. In addition, if $Y_\alpha \in \mathcal J^+(\A_\alpha)$, we can take $Y\subseteq Y_\alpha$:

Indeed, if $\{\gamma<\alpha : |a_\gamma\cap Y_\alpha|=\omega \}$ is finite, take $Y\doteq Y_\alpha \setminus \bigcup\{a_\gamma: \gamma<\alpha \wedge |a_\gamma\cap Y_\alpha|=\omega \}$. Otherwise, note that $\mathcal{B}=\{ a_\gamma\cap Y_\alpha:\gamma<\alpha \wedge |a_\gamma\cap Y_\alpha|=\omega\}$ is an almost disjoint family in $Y_\alpha$. Since $|\mathcal{B}|<\mathfrak a$, there exists $Y\subseteq Y_\alpha$ almost disjoint from each element of $\mathcal{B}$.

Since $\mathfrak s =\mathfrak c$, $\{X_\gamma\cap Y: \gamma\leq \alpha \}$ is not a splitting family in $Y$. Thus, there exists $a_\alpha\subseteq Y$ such that for all $\gamma \leq \alpha$, $a_\alpha\cap X_\gamma =^*\emptyset$ or $a_\alpha\subseteq^* X_\gamma$.

 Notice that $\A$ is an almost disjoint family extending $\mathcal A'$ by (2) and (7)

We show that for every infinite $Y\subseteq \omega$, either $Y \notin \mathcal J^+(\A)$ or there exists $\alpha<\mathfrak c$ such that $a_\alpha\subseteq Y$, thus proving $\A$ is MAD and completely separable. If $Y\in \mathcal J^+(\A)$, let $\alpha$ be such that $Y=Y_\alpha$. Then $Y\in \mathcal J^+(\A_\alpha)$, thus, by 6., $a_\alpha\subseteq Y_\alpha$.

Finally, we prove that $\mathcal A$ is $(\aleph_0, <\!\mathfrak c)$-separated. Given an infinite countable set $\mathcal B\subseteq \mathcal A$ and $\mathcal C\in [\mathcal A]^{<\mathfrak c}$ such that $\mathcal B\cap\mathcal C=\emptyset$, let $B=\{\alpha<\mathfrak c: a_\alpha \in \mathcal B\}$ and $C=\{\alpha<\mathfrak c: a_\alpha \in \mathcal C\}$. Notice that $B$ is infinite and countable, $|C|<\mathfrak c$, and $B\cap C=\emptyset$. Let $\alpha_0=\sup C$, which is less than $\mathfrak c$ since $\mathfrak c=\mathfrak b$ is regular, and let $\alpha>\alpha_0$ be such that $B_\alpha=B$. In particular, $B=B_\alpha\subseteq \alpha$ and $C\subseteq \alpha\setminus B_\alpha$. By (3), for all $b \in \mathcal B$, $b\subseteq^*X_\alpha$.

By (4), for all $c\in \mathcal C$, $c\cap X_\alpha=^*\emptyset$ 

By (3) and (4) together by using $\alpha$ in the place of $\beta$, we see that for every $\gamma<\alpha$, $a_\gamma\subseteq^* X_\alpha$ or $a_\gamma\cap X_\alpha=^*\emptyset$. If $\gamma>\alpha$, we apply (5) for this $\gamma$ and $\alpha$ in the place of $\beta$ to conclude that $a_\gamma\subseteq^* X_\alpha$ or $a_\gamma\cap X_\alpha=^*\emptyset$.
\end{proof}

It is well know that the cardinal characteristic $\mathfrak{par}$, as defined in $\cite{blass2010}$, equals the minimum of $\mathfrak b, \mathfrak s$. Thus, the previous theorem could have its hypothesis replaced by $\mathfrak{par}=\mathfrak c$.

\section{Conclusion}

We have answered Question 4.4 of \cite{paulsergio} by providing a counter example in $\beta \omega$ and partially answered Question 4.2 of \cite{paulsergio} by providing an example by using forcing. We have shown that an almost disjoint family is semi-normal iff it is weakly separated, thus, for weakly separated almost disjoint families normality and almost-normality are the same. However, Question 4.2 remains open. We may define $\mathfrak{an}$ as the least cardinality of an almost-normal almost disjoint family which is not normal. We don't know if this number is well defined in ZFC, however, if there is such an almost disjoint family, it follows that $\mathfrak{ap}\leq\mathfrak{an}$. We may refine Question 4.2. as follows:

\begin{question}
Is $\mathfrak{an}$ well defined in ZFC? If there is an almost-normal not normal almost disjoint family, does $\mathfrak{an}=\mathfrak{ap}$ hold?
\end{question}

Recall that in \cite{paulsergio} it was proven that almost-normal almost disjoint families are strongly $\aleph_0$-separated. Here we have defined the concept of strongly $(\aleph_0, <\mathfrak c)$-almost disjoint families and we have proved that strongly $(\aleph_0, <\mathfrak c)$-separation property does not hold for all almost-normal almost disjoint families, at least consistently.
However, the relation between these concepts is not fully understood. Thus, we ask:

\begin{question}
Are almost-normal almost disjoint families strongly $(\aleph_0,<\!\mathfrak c)$-separated? 
\end{question}

\begin{question}
Does CH imply that strongly $\aleph_0$-separated almost disjoint families are almost-normal?
\end{question}

\section{Acknowledgements}

The first author was funded by FAPESP (Fundação de Amparo à Pesquisa do Estado de São Paulo, process number 2017/15502-2). The second author was funded by CNPq (Conselho Nacional de Desenvolvimento Científico e Tecnológico, process number 141881/2017-8).

The authors would like to thank Sergio A. Garcia-Balan and Professor Paul Szeptycki for reading preliminary versions of this paper and for making suggestions. \change{We also would like to thank the referee for his or her suggestions}.


\bibliographystyle{plain}


\begin{thebibliography}{10}

\bibitem{alshammari2020quasi}
Ibtesam~Eid AlShammari and Lutfi Kalantan.
\newblock Quasi-normality of Mrowka spaces.
\newblock {\em European Journal of Pure and Applied Mathematics},
  13(3):697--700, 2020.

\bibitem{blass2010}
Andreas Blass.
\newblock Combinatorial cardinal characteristics of the continuum.
\newblock In {\em Handbook of set theory}, pages 395--489. Springer, 2010.

\bibitem{Brendle}
J{\"o}rg Brendle.
\newblock Dow’s principle and {Q}-sets.
\newblock {\em Canadian Mathematical Bulletin}, 42(1):13--24, 1999.

\bibitem{Adow}
Alan Dow.
\newblock On compact separable radial spaces.
\newblock {\em Canadian Mathematical Bulletin}, 40(4):422--432, 1997.

\bibitem{engelking1977general}
Ryszard Engelking et~al.
\newblock {\em General topology.}
\newblock Heldermann Verlag, 1977.

\bibitem{OnQsets}
William~G Fleissner and Arnold~W Miller.
\newblock On {$Q$} sets.
\newblock {\em Proceedings of the American Mathematical Society},
  78(2):280--284, 1980.

\bibitem{galvin}
Fred Galvin and Petr Simon.
\newblock A \v{C}ech function in {ZFC}.
\newblock {\em Fundamenta Mathematicae}, 193(2):181--188, 2007.

\bibitem{paulsergio}
Sergio~A Garcia-Balan and Paul~J Szeptycki.
\newblock Weak normality properties in {$\Psi$}-spaces.
\newblock {\em arXiv preprint arXiv:2007.05844}, 2020.

\bibitem{genericMAD}
Osvaldo Guzm{\'a}n-Gonz{\'a}lez, Michael Hru{\v{s}}{\'a}k, Carlos~Azarel
  Mart{\'\i}nez-Ranero, and Ulises~Ariet Ramos-Garc{\'\i}a.
\newblock Generic existence of {MAD} families.
\newblock {\em The Journal of Symbolic Logic}, 82(1):303--316, 2017.

\bibitem{HernandezHernandez2018}
F.~Hern{\'a}ndez-Hern{\'a}ndez and M.~Hru{\v{s}}{\'a}k.
\newblock {\em Topology of Mr{\'o}wka-Isbell Spaces}, pages 253--289.
\newblock Springer International Publishing, Cham, 2018.

\bibitem{qsetsnormality}
Fernando Hern{\'a}ndez-Hern{\'a}ndez and Michael Hru{\v{s}}{\'a}k.
\newblock {$Q$}-sets and normality of {$\Psi$}-spaces.
\newblock {\em Topology Proc}, 29:155--165, 2005.

\bibitem{Hrusak2014}
Michael Hru{\v{s}}{\'a}k.
\newblock {\em Almost Disjoint Families and Topology}, pages 601--638.
\newblock Atlantis Press, Paris, 2014.

\bibitem{kalantan}
Lufti Kalantan and Fateh Allahabi.
\newblock Almost and mild normality.
\newblock Avaiable at
  \url{https://www.kau.edu.sa/Files/130/Researches/19623_AlmostAndMildNormality_Kalantan_Allahabi.pdf}.

\bibitem{kalantan2008}
Lutfi~N Kalantan.
\newblock $\pi$-normal topological spaces.
\newblock {\em Filomat}, 22(1):173--181, 2008.

\bibitem{kunnen1980}
K~Kunen.
\newblock Set theory: an introduction to independence proofs, volume 102 of
  studies in logic and the foundations of mathematics, 1980.

\bibitem{kunen2011set}
K.~Kunen.
\newblock {\em Set Theory}.
\newblock Studies in logic. College Publications, 2011.

\bibitem{shchepin1972real}
EV~Shchepin.
\newblock Real functions and near-normal spaces.
\newblock {\em Siberian Mathematical Journal}, 13(5):820--830, 1972.

\bibitem{singal1970almost}
MK~Singal and Shashi~Prabha Arya.
\newblock Almost normal and almost completely regular spaces.
\newblock {\em Glasnik Mat}, 5(25):141--152, 1970.

\bibitem{SingalArya}
M.K. Singal and S.P. Arya.
\newblock On almost normal and almost completely regular spaces.
\newblock {\em Glasnik Mat.}, 5(1):141–152, 1970.

\bibitem{singal1973mildly}
MK~Singal and Asha~Rani Singal.
\newblock Mildly normal spaces.
\newblock {\em Kyungpook Mathematical Journal}, 13(1):27--31, 1973.

\bibitem{tall1977set}
Franklin~D Tall.
\newblock Set-theoretic consistency results and topological theorems concerning
  the normal moore space conjecture and related problems, 1977.

\bibitem{zaitsavrussian}
Zaitsav. V.
\newblock On certain classes of topological spaces and their
  bicompactifications.
\newblock {\em Dokl.Akad. Nauk SSSR}, 178:778 -- 779, 1968.

\end{thebibliography}

\end{document}